\documentclass{article}
\usepackage{natbib}
\usepackage[dvipdfmx]{graphicx}
\usepackage{amsmath}
\usepackage{amsthm}
\usepackage{amsfonts}
\usepackage{amssymb}
\usepackage{latexsym}
\usepackage{color}
\def\qed{\hfill $\Box$}
\makeatletter
\renewenvironment{proof}[1][\proofname]{\par
  \normalfont
  \topsep6\p@\@plus6\p@ \trivlist
  \item[\hskip\labelsep{\bfseries #1}\@addpunct{\bfseries.}]\ignorespaces
}{
  \endtrivlist
}
\renewcommand{\proofname}{Proof}
\makeatother
\begin{document}
\theoremstyle{plain}
\newtheorem{defi}{Definition}[section]
\newtheorem{lem}[defi]{Lemma}
\newtheorem{thm}[defi]{Theorem}
\newtheorem{rem}[defi]{Remark}
\newtheorem{fact}[defi]{Fact}
\newtheorem{ex}[defi]{Example}
\newtheorem{prop}[defi]{Proposition}
\newtheorem{cor}[defi]{Corollary}
\newtheorem{condition}[defi]{Condition}
\newtheorem{assumption}[defi]{Assumption}
\newcommand{\sign}{\mathop{\rm sign}}
\newcommand{\conv}{\mathop{\rm conv}}
\newcommand{\argmax}{\mathop{\rm arg~max}\limits}
\newcommand{\argmin}{\mathop{\rm arg~min}\limits}
\newcommand{\argsup}{\mathop{\rm arg~sup}\limits}
\newcommand{\arginf}{\mathop{\rm arg~inf}\limits}
\newcommand{\diag}{\mathop{\rm diag}}
\newcommand{\minimize}{\mathop{\rm minimize}\limits}
\newcommand{\maximize}{\mathop{\rm maximize}\limits}
\title{Cox's proportional hazards model with a high-dimensional and sparse regression parameter}
\author{Kou Fujimori
\\
\\
Waseda University}
\date{}
\maketitle
\begin{abstract}
This paper deals with the proportional hazards model proposed by D.\ R.\ Cox in a 
high-dimensional and sparse setting for a regression parameter. 
To estimate the regression parameter, the Dantzig selector is applied.
The variable selection consistency of the Dantzig selector for the model will be proved.
This property enables us to reduce the dimension of the parameter and to 
construct asymptotically normal estimators for the regression parameter and the cumulative baseline hazard function.
\end{abstract}
\section{Introduction}
\label{}
The proportional hazards model, which was proposed by \cite{key cox}
is one of the most commonly used models for survival analysis.
In a fixed dimensional setting, $i.e.$, the case where the number of covariates $p$ is fixed,
\cite{key gill} proved that the maximum partial likelihood estimator for
regression parameter has the consistency and the asymptotic normality. 
Besides, they discussed the asymptotic property of the Breslow estimator for cumulative 
baseline hazard function.

Recently, many researchers are interested in a high-dimensional and sparse setting for a regression parameter, that is, the case where $p \gg n$ and the number $S$ of nonzero components in the true value is relatively small. In this setting, several kinds of estimation methods have been proposed for various regression-type models. 
Especially, the penalized methods such as Lasso 
(\cite{key tibshirani97}, \cite{key huang}, \cite{key bradic} and others) have been well studied.
In particular, \cite{key huang} derived oracle inequalities of the Lasso estimator for the 
proportional hazards model, which means the Lasso estimator satisfies the consistency even in a high-dimensional setting. 
\cite{key bradic} considered the general penalized estimators including Lasso, SCAD and others 
and proved that the estimators satisfies the consistency and the asymptotic normality.
On the other hand, the Dantzig selector, which was proposed by \cite{key tao} for the linear regression model, is also applied to the proportional hazards model by 
\cite{key antoniadis}, who dealt with the $l_2$ consistency of the estimator. 
\cite{key Fujimori} extended the consistency results of the Dantzig selector for the model to the $l_q$ consistency for every $q \in [1,\infty]$ 
by a method similar to that of \cite{key huang}.
However, the asymptotic normalities of estimators for high-dimensional regression parameter and the Breslow estimator have not yet been studied up to our knowledge. 

In this paper, we will focus on the asymptotic normalities of estimators in a high-dimensional setting. To discuss this problem, we need to consider the dimension reduction of the regression parameter.
We will show that the Dantzig selector has variable selection consistency, which enables us 
to reduce the dimension.
Then, we will construct a new maximum partial likelihood estimator by using the variable selection consistency result and 
show that this estimator has the asymptotic normality.
In addition, we will prove that a Breslow type estimator, which is obtained by using the 
maximum partial likelihood estimator after dimension reduction, satisfies the asymptotic normality.

This paper is organized as follows. 
The model setup, some regularity conditions and matrix conditions to deal with a high-dimensional and sparse setting are introduced in Section 2.
In Section 3, we prove the asymptotic properties of the high-dimensional regression parameter, that is, the variable selection consistency of the Dantzig selector and the asymptotic normality of the maximum partial likelihood estimator after dimension reduction.
The asymptotic property of the Breslow type estimator is established in Section 4.

Throughout this paper, we denote by $\|\cdot\|_q$ the $l_q$ norm of vector for every $q \in[1,\infty]$, $i.e.$ for $v = (v_1,v_2,\ldots,v_p)^T \in \mathbb{R}^p$ we denote: 
\begin{align*}
\|v\|_q &= \left(\sum \limits_{j=1}^p |v_j|^q \right)^{\frac{1}{q}},\quad q < \infty ; \\
\|v\|_{\infty} &= \sup \limits_{1 \leq j \leq p} |v_j|.
\end{align*}
In addition, for a $m \times n$ matrix $A$, where $m,\ n \in \mathbb{N}$, we define $\|A\|_{\infty}$ by
\[\|A\|_{\infty} := \sup \limits_{1 \leq i \leq m} \sup \limits_{1 \leq j \leq n} |A_i^j|,\]
where $A_i^j$ denotes the $(i,j)$-component of the matrix $A$.
For a vector $v \in \mathbb{R}^p$, and an index set $T\subset \{1,2,\ldots,p\}$, 
we write $v_T$ for the $|T|$-dimensional sub-vector of $v$ restricted by the index set $T$, 
where $|T|$ is the number of elements in the set $T$.
Similarly, for a $p \times p$ matrix $A$ and index sets $T,T' \subset \{1,2,\ldots,p\}$, 
we define the $|T| \times |T'| $ sub-matrix $A_{T,T'}$ by 
\[
A_{T,T'} := (A_{i}^j)_{i \in T, j \in T'}.
\]
\section{Preliminaries}
\subsection{Model setup}
Let $T_i$ be a survival time and $C_i$ a censoring time of $i$-th individual for every 
$i =1,2,\ldots,n$, which are positive real valued random variables on a probability space 
$(\Omega,\mathcal{F},P)$.
Assume that each $i$-th individual has an $\mathbb{R}^p$-valued 
covariate process $\{Z_i(t)\}_{t \in [0,1]}$, and that the survival time $T_i$ is conditionally independent of the censoring time $C_i$ given $Z_i(t)$.
Moreover, we assume that $T_i$'s never occur simultaneously.
For every $n \in \mathbb{N}$ and $t \in [0,1]$, 
we observe $\{(X_i,D_i,Z_i(t))\}_{i=1}^n$, where $X_i := T_i \land C_i$ and 
$D_i := 1_{\{T_i \leq C_i\}}$. We define the counting process $\{N_i(t)\}_{t \in [0,1]}$ and 
$\{Y_i(t)\}_{t \in [0,1]}$ for every 
$i = 1,2,\ldots,n$ as follows:
\[
N_i(t) := 1_{\{t \leq X_i,D_i = 1\}}, \quad Y_i(t) := 1_{\{X_i \geq t\}}, \quad t \in [0,1] 
\]
Let $\{\mathcal{F}_t\}_{t \in [0,1]}$ be the filtration defined as follows:
\begin{align*}
\mathcal{F}_t  &:= \sigma \{ N_i(u),\ Y_i(u),\ Z_i(u);\ 0 \leq u \leq t,\ i = 1,2,\ldots,n \}.
\end{align*}
Suppose that $\{Z_i(t)\}_{t \in [0,1]}$, $i = 1,2,\ldots,n$ are predictable processes.
In Cox's proportional hazards model, it is assumed that each $\{N_i(t)\}_{t \in [0,1]}$ for every 
$i = 1,2,\ldots,n$ has the following intensity:
\[
\lambda_i(t) := Y_i(t) \lambda_0(t) \exp(\beta_0^T Z_i(t)),\quad t \in [0,1],
\]
where $\lambda_0(t)$ is the unknown deterministic baseline hazard function and 
$\beta_0 \in \mathbb{R}^p$ is the unknown regression parameter.
Then, we have that the following process $\{M_i(t)\}_{t \in [0,1]}$ for every $i = 1,2,\ldots,n$ 
is a square integrable martingale:
\[
M_i(t):= N_i(t) - \int_0^t \lambda_i(s) ds,\quad t \in [0,1].
\]
Note that predictable variation process of $\{M_i(t)\}_{t \in [0,1]}$ is given by:
\[
\langle M_i, M_i \rangle (t) = \int_0^t \lambda_i(s) ds, \quad t \in[0,1]
\]
and
\[
\langle M_i, M_j \rangle (t) = 0 ,\quad i \not= j,\ t \in [0,1].
\]
Hereafter, we write $\Lambda_0$ for the cumulative baseline hazard function, $i.e.$, 
\[
\Lambda_0(t) := \int_0^t \lambda_0(s) ds,\quad t \in [0,1].
\]
The aim of this paper is to estimate the regression parameter $\beta_0$ and the
cumulative baseline hazard $\Lambda_0$ in a high-dimensional and sparse setting for 
$\beta_0$, $i.e.$, 
\[
p = p_n \gg n,\quad S := |T_0| \ll n,
\]
where $T_0 := \{j;\beta_0^j \not = 0\}$ is the support index set of the true value.
To estimate $\beta_0$, we use Cox's $\log$-partial likelihood which is given by;
\[
C_n(\beta) := \sum_{i = 1}^n \int_0^1 \{\beta^T Z_i(t) - \log S_n^{(0)} (\beta,t) \} dN_i(t),
\]
where 
\[
S_n^{(0)} (\beta,t) := \sum_{i =1}^n Y_i (t) \exp(\beta^T Z_i(t)).
\]
Put $l_n(\beta) = C_n(\beta)/n$.
We write $U_n(\beta)$ for the gradient of $l_n(\beta)$ and 
$J_n(\beta)$ for the Hessian of $-l_n(\beta)$, $i.e.$, 
\[
U_n(\beta) = \frac{1}{n} \sum_{i=1}^n \int_0^1 \left\{Z_i(t) - \frac{S_n^{(1)}}{S_n^{(0)}} (\beta,t)\right\}dN_i(t)
\]
and
\[
J_n(\beta) = \frac{1}{n} \sum_{i = 1}^n \int_0^1 \left\{\frac{S_n^{(2)}}{S_n^{(0)}}(\beta,t)
-\left(\frac{S_n^{(1)}}{S_n^{(0)}}\right)^{\otimes 2} (\beta,t) \right\}dN_i(t),
\]
where
\[
S_n^{(1)} (\beta,t) := \sum_{i =1}^n Z_i(t) Y_i(t) \exp(\beta^T Z_i(t))
\]
and 
\[
S_n^{(2)}(\beta,t) := \sum_{i=1}^n Z_i(t)^{\otimes 2} Y_i (t) \exp(\beta^T Z_i(t)).
\]
Note that $U_n(\beta_0)$ is a terminal value of the following square integrable martingale:
\[
U_n(\beta_0,t) := \frac{1}{n} \sum_{i=1}^n \int_0^t \left\{Z_i(s) - \frac{S_n^{(1)}}{S_n^{(0)}} (\beta,s)\right\}dM_i(s).
\]
\subsection{Regularity conditions and matrix conditions}
We assume the following conditions.
\begin{assumption}\label{regularity}
\begin{description}
\item[$(i)$]
The true value $\beta_0$ satisfies that $\|\beta_0\|_1 < \infty$. Moreover, there exists a global constant $C >0$ such that
\[\inf_{j \in T_0} |\beta_0^j| > C.
\]
\item[$(ii)$]
The covariate processes $\{Z_i(t)\}_{t \in [0,1]}$, $i = 1,2,\ldots,n$, are 
uniformly bounded, $i.e.$, there exists global constant $K_1 > 0$ such that 
\[
\sup_{t \in [0,1]} \sup_{i} \|Z_i(t)\|_\infty < K_1.
\]
\item[$(iii)$]
The baseline hazard function $\lambda_0$ is integrable, $i.e.$, 
\[
\int_0^1 \lambda_0(t) dt < \infty.
\]
\item[$(iv)$]
For every $n \in \mathbb{N}$, there exist deterministic $\mathbb{R}$-valued function $s_n^{(0)}(\beta,t)$, 
$\mathbb{R}^{p_n}$ valued function $s_n^{(1)}(\beta,t)$ and $\mathbb{R}^{p_n \times p_n}$-
valued function $s_n^{(2)}(\beta,t)$ which satisfy the following conditions:
\[
\sup_{\beta} \sup_{t \in [0,1]} 
\left\|\frac{1}{n} S_n^{(l)} (\beta,t) - s_n^{(l)}(\beta,t) \right\|_\infty \rightarrow^p 0, \quad l = 0,1,2
\] 
as $n \rightarrow \infty$.
\item[$(v)$]
The functions $s_n^{(l)}(\beta,t)$, $l=0,1,2$, satisfy the following conditions:
\[
\limsup_{n \rightarrow \infty} \sup_{\beta} \sup_{t \in [0,1]} \|s_n^{(l)}(\beta,t)\|_\infty < \infty
,\quad l = 0,1,2.
\]
\[
\liminf_{n \rightarrow \infty} \inf_{\beta} \inf_{t \in [0,1]} s_n^{(l)}(\beta,t) >0.
\]
\item[$(vi)$]
For every $\beta$, the following $p_n \times p_n$ matrix $I_n(\beta)$ is nonnegative definite:
\[
I_n(\beta):= \int_0^1 \left[\frac{s_n^{(2)}}{s_n^{(0)}} (\beta,t) - \left(\frac{s_n^{(1)}}{s_n^{(0)}}\right)^{\otimes 2} (\beta,t) \right] s_n^{(0)}(\beta_0,t) \lambda_0(t) dt.
\] 
\item[$(vii)$]
For every $\epsilon > 0$, it holds that
\[
\sum_{i = 1}^n \int_0^1 
\left\| \xi_{n T_0, i}\right\|_2^2
1_{\{\|\xi_{n T_0,i}\|^2\} > \epsilon} Y_i(t) \exp(\beta_0^T Z_i(t)) \lambda_0(t) dt
\rightarrow^p 0,
\]
where
\[
\xi_{n T_0,i} := \frac{1}{\sqrt{n}}\left\{ Z_{i T_0}(t) - \frac{S_{n T_0}^{(1)}}{S_n^{(0)}}(\beta_{0 T_0},t) \right\}.
\]
\end{description}
\end{assumption}
Note that the condition $(vii)$ ensures that Lindeberg's condition holds.
Recalling that $T_0 = \{j; \beta_0^j \not=0\}$ is the support index set 
of the true value $\beta_0$,
we introduce the following factor for the matrix $I_n(\beta_0)$.
\begin{defi}
Define the set $C_{T_0} \subset \mathbb{R}^{p_n}$ as follows:
\[
C_{T_0} := \{h \in \mathbb{R}^{p_n} ; \|h_{T_0^c}\|_1 \leq \|h_{T_0}\|_1 \}.
\]
\begin{description}
\item[Compatibility factor:]
\[
\kappa (T_0;I_n(\beta_0)) := \inf_{0 \not=h \in C_{T_0}} \frac{S^{\frac{1}{2}} (h^T I_n(\beta_0)h)^{\frac{1}{2}}}{\|h_{T_0}\|_1}.
\]
\end{description}
\end{defi}
The matrix factor like this can be seen in many papers which deal with high-dimensional and sparse setting. See, e.g., \cite{key bickel-ritov-tsybakov}, \cite{key buhlmann} and 
\cite{key huang} for the details.
We assume the following condition for $I_n(\beta_0)$.
\begin{assumption}\label{matrix}
The compatibility factor $\kappa(T_0;I_n(\beta_0))$ is asymptotically positive, $i.e.$,
\[
\liminf_{n \rightarrow \infty} \kappa(T_0;I_n(\beta_0)) >0.
\]
\end{assumption}
\section{The estimator for the regression parameter}
\subsection{The Dantzig selector for the proportional hazards model}
Now, we define the estimator for the regression parameter $\beta_0$ by the Dantzig selector
for the proportional hazards model given by:
\begin{equation}\label{Dantzig selector}
\hat{\beta}_n := \arg \min_{\beta \in \mathcal{B}_n} \| \beta \|_1,\quad 
\mathcal{B}_n := \{\beta \in \mathbb{R}^{p_n} ; \|U_n(\beta_0)\|_\infty \leq \gamma\},
\end{equation}
where $\gamma$ is a tuning parameter.
This type of estimator was proposed by \cite{key antoniadis} and was further  discussed by \cite{key Fujimori}.
\subsection{The $l_q$ consistency of the Dantzig selector}
In this subsection, we discuss the consistency of the estimator $\hat{\beta}_n$ in the sense of $l_q$-norm for every $q \in [1,\infty]$.
Assume that $p_n$ and $\gamma = \gamma_{n,p_n}$ satisfy the following conditions:
\[
\log p_n = O(n^\zeta),\quad \gamma_{n,p_n} = O(n^{-\alpha}\log p_n),
\]
where $0 < \zeta < \alpha \leq 1/2$ are constants. Suppose that the sparsity $S$ is fixed constant which does not depend on $n$.
Moreover, we define the random sequence $\epsilon_n$ by:
\[
\epsilon_n := \|I_n(\beta_0) - J_n(\beta_0)\|_\infty.
\]
Then, we can show that $\epsilon_n \rightarrow^p 0$ (see \cite{key Fujimori}).
\begin{thm}\label{l_q}
Under the Assumption \ref{regularity} and \ref{matrix},
the it holds for global positive constants $K_2$ that
\[
\lim \limits_{n \rightarrow \infty}P\left(\|\hat{\beta}_n - \beta_0 \|_1 \geq \frac{4K_2 S \gamma_{n,p_n}}{\kappa^2(T_0;I_n(\beta_0)) - 4S \epsilon_n}\right) = 0.
\]
In particular, it holds for every $q \in [1,\infty]$ that 
$\|\hat{\beta}_n - \beta_0\|_q \rightarrow^p 0$ as $n \rightarrow \infty$.
\end{thm}
The proof of Theorem \ref{l_q} can be seen in \cite{key Fujimori}.
\subsection{The variable selection consistency of the Dantzig selector}
The aim of this subsection is to show that $\hat{\beta}_n$ selects non-zero components of
$\beta_0$ correctly.
To do this, we define the following estimator for the support index set $T_0$ of the true value 
$\beta_0$:
\[
\hat{T}_n := \{j; |\hat{\beta}_n^j| > \gamma_{n,p_n}\}.
\]
The estimator similar to $\hat{T}_n$ can be seen in \cite{key Fujimori3} which consider a linear model of diffusion processes in a high-dimensional and sparse setting.
The following theorem states that $\hat{\beta}_n$ has a variable selection consistency.
\begin{thm}\label{selection}
Under the Assumption \ref{regularity} and \ref{matrix}, it holds that
\[
\lim_{n \rightarrow \infty} P(\hat{T}_n = T_0) = 1.
\]
\end{thm}
\begin{proof}
Note that $\|\hat{\beta}_n - \beta_0\|_\infty \leq \|\hat{\beta}_n - \beta_0\|_1$ 
and that the sparsity $S$ is assumed to be fixed.
We have that 
\[
\lim_{n \rightarrow \infty} 
P\left(\|\hat{\beta}_n - \beta_0\|_\infty > \gamma_{n,p_n}\right) = 0
\]
by the $l_1$ bound from Theorem \ref{l_q} $(i)$.
Therefore, it is sufficient to show that the next inequality
\[
\|\hat{\beta}_n - \beta_0 \|_\infty \leq \gamma_{n,p_n}
\]
implies that 
\[\hat{T}_n = T_0.\]
For every $j \in T_0$, it follows from the triangle inequality that 
\[|\beta_0^j| - |\hat{\beta}_n^j| \leq |\hat{\beta}_n^j - \beta_0^j| \leq \gamma_{n,p_n}.\]
Then, we have that
\[
|\hat{\beta}_n^j| \geq |\beta_0^j|-\gamma_{n,p_n} > \gamma_{n,p_n}
\]
for sufficiently large $n$, which implies that $T_0 \subset \hat{T}_n$. 
On the other hand, for every $j \in T_0^{c}$, we have that 
\[
|\hat{\beta}_n^j - \beta_0^j| = |\hat{\beta}_n^j| \leq \gamma_{n,p_n}
\]
since it holds that $\beta_0^j = 0$.
Then, we can see that $j \in \hat{T}_n^{c}$ which implies that 
$\hat{T}_n \subset T_0$.
We thus obtain the conclusion.
\qed
\end{proof}
\subsection{The maximum partial likelihood estimator for the regression parameter after dimension reduction}
Using the set $\hat{T}_n$, we construct a new estimator $\hat{\beta}_n^{(2)}$ by the solution to the next equation:
\begin{equation}\label{2nd estimator}
U_n(\beta_{\hat{T}_n}) = 0, \quad \beta_{\hat{T}_n^c} = 0.
\end{equation}
We prove the asymptotic normality of $\hat{\beta}_n^{(2)}$.
In this subsection, we assume that the following $S \times S$ matrix $\mathcal{I}$ is positive definite:
\[
\mathcal{I} := \int_0^1 \left[
\frac{s^{(2)}}{s^{(0)}}(\beta_{0T_0},t) - \left(\frac{s^{(1)}}{s^{(0)}}\right)^{\otimes 2}(\beta_{0T_0},t)
\right]  \lambda_0(s) ds,
\]
where
\[
s^{(0)}(\beta_{0 T_0},t) := s_n^{(0)}(\beta_{0T_0},t),
\]
\[
s^{(1)}(\beta_{0 T_0},t) := s_{n T_0}^{(1)}(\beta_{0 T_0},t)
\]
and 
\[
s^{(2)}(\beta_{0 T_0},t) := s_{n T_0,T_0}^{(2)} (\beta_{0 T_0},t).
\]
The following theorem states that this estimator $\hat{\beta}_n^{(2)}$ satisfies $l_1$ consistency.
\begin{thm}\label{l_1}
Under Assumption \ref{regularity} and \ref{matrix}, it holds that
\[
\|\hat{\beta}_n^{(2)} - \beta_0\|_1 \rightarrow^p 0
\]
as $n \rightarrow \infty$.
\end{thm}
\begin{proof}
We have that 
\[
\|\hat{\beta}_n^{(2)} - \beta_0\|_1 = \|\hat{\beta}_{n T_0}^{(2)} - \beta_{0 T_0}\|_1 + 
\|\hat{\beta}_{n T_0^c}^{(2)}\|_1.
\]
It follows from Lemma 3.1 of \cite{key gill} that the first term of right-hand side is $o_p(1)$ since the sparsity $S$ is assume to be fixed. 
Moreover, we have that
\[
\|\hat{\beta}_{n T_0^c}^{(2)}\|_1 1_{\{\hat{T}_n = T_0\}} = 0
\]
by the definition of $\hat{\beta}_n^{(2)}$.
Noting that $1_{\{\hat{T}_n = T_0\}} \rightarrow^p 1$, we obtain the conclusion by using 
Slutsky's theorem.
\qed
\end{proof}
To show the asymptotic normality of $\hat{\beta}_n^{(2)}$, we need to prove the next lemma.
\begin{lem}\label{fischer}
For every random sequence $\{\beta_n^*\}_{n \in \mathbb{N}}$ which satisfies that
\[
\|\beta_n^* - \beta_0\|_1 \rightarrow^p 0
\]
as $n \rightarrow \infty$, it holds that
\[
\|J_n(\beta_n^*) - I_n(\beta_0)\|_\infty = o_p(1).
\]
\end{lem}
\begin{proof}
We have 
for every $l = 0,1,2$ and $t \in [0,1]$ that 
\begin{eqnarray*}
\lefteqn{
\frac{1}{n}\|S_n^{(l)} (\beta_n^*,t) - S_n^{(l)}(\beta_0,t)\|_\infty} \\
&\leq&
\frac{1}{n}\left\|\sum_{i=1}^n Y_i(t) Z_i(t)^{\otimes l}(t) \exp(\beta_0^T Z_i(t))
\left\{\exp[\|Z_i(t)\|_\infty \|\beta_n^* - \beta_0\|_1] -1\right\}\right\|_\infty \\
&\leq&
K_1 \exp(K_1 \|\beta_0\|_1) |\exp(K_1 \|\beta_n^* - \beta_0\|_1) -1|.
\end{eqnarray*}
The right-hand side of this inequality converges to $0$ in probability when $\|\beta_n^* - \beta_0\|_1 \rightarrow^p 0$ as $n \rightarrow \infty$.
Then, we obtain the conclusion by a similar way to the proof in \cite{key gill}.
\qed
\end{proof}
Then, we can prove the asymptotic normality in the following sense by a similar way to that in \cite{key gill}.
\begin{thm}
Under Assumption \ref{regularity} and \ref{matrix}, it holds that 
\[
\sqrt{n}(\hat{\beta}_{n \hat{T}_n}^{(2)} - \beta_{0 T_0}) 1_{\{\hat{T}_n = T_0\}}
\rightarrow^d N(0,\mathcal{I}^{-1}).
\]
\end{thm}
\begin{proof}
It follows from the Taylor expansion and Lemma \ref{fischer} that
\[
\sqrt{n} (\hat{\beta}_{n \hat{T}_n}^{(2)} - \beta_{0 T_0}) 1_{\{\hat{T}_n = T_0\}}
= \mathcal{I}^{-1} \sqrt{n} U_{n T_0}(\beta_{0 T_0}) 1_{\{\hat{T}_n = T_0\}} + o_p(1).
\]
Using the martingale central limit theorem, we can see that 
\[
\sqrt{n} U_{n T_0}(\beta_{0 T_0}) \rightarrow^d N(0, \mathcal{I}).
\]
Since Theorem \ref{selection} implies that $1_{\{\hat{T}_n = T_0\}} \rightarrow^p 1$ as $n \rightarrow \infty$, we obtain the conclusion by using Slutsky's theorem.
\qed
\end{proof}
\section{The estimator for the cumulative baseline hazard function}
We define the estimator for $\Lambda_0(t)$ by the following Breslow type estimator:
\[
\hat{\Lambda}(t) := \int_0^t \frac{d\bar{N}(s)}{\sum_{i = 1}^n Y_i(s) \exp(\hat{\beta}_n^{(2) T} Z_i(s))},\quad t \in [0,1],
\]
where $\hat{\beta}_n^{(2)}$ is defined by the equation (\ref{2nd estimator}).
We discuss the asymptotic property of $\hat{\Lambda}$ in this section.
For every $t \in [0,1]$, we have that 
\[
\sqrt{n}\{\hat{\Lambda}(t) - \Lambda_0(t)\}
= (I) + (II) + (III),
\]
where
\[
(I) = \sqrt{n} \int_0^t \left\{\frac{1}{S_n^{(0)}(\hat{\beta}_n^{(2)},s)} - \frac{1}{S_n^{(0)}(\beta_0,s)}\right\} d\bar{N}(s),
\]
\[
(II) = \sqrt{n} \left\{\int_0^t \frac{d\bar{N}(s)}{S_n^{(0)}(\beta_0,s)} - 
\int_0^t \lambda_0(s) 1_{\{\sum_{i=1}^n Y_i(s) >0\}}
\right\}
\]
and
\[
(III) = \sqrt{n}\left\{\int_0^t \lambda_0(s) 1_{\{\sum_{i=1}^n Y_i(s) >0\}}-\Lambda_0(t)
\right\}.
\]
The third term $(III)$ is asymptotically negligible because it follows from Assumption \ref{regularity} that 
\[
\lim_{n \rightarrow \infty}P\left(\left\{\int_0^t \lambda_0(s) 1_{\{\sum_{i=1}^n Y_i(s) >0\}}-\Lambda_0(t)
\right\} = 0\right) = 1.
\]
Moreover, we have that $(II)$ equals to the following process $\{W_n(t)\}_{t \in [0,1]}$:
\[
W_n(t) = \sqrt{n} \int_0^t \frac{d\bar{M}(s)}{S_n^{(0)}(\beta_0,s)},
\]
which is a square integrable martingale.
Using the Taylor expansion, we have that
\[
(I) = H_n(\beta_n^*,t)^T (\hat{\beta}_n^{(2)}-\beta_0),
\]
where 
\[
H_n(\beta_n^*,t) := -\int_0^t \frac{S_n^{(1)}}{\{S_n^{(0)}\}^2}(\beta_n^*,s) d\bar{N}(s)
\]
and $\beta_n^*$ lies between $\hat{\beta}_n^{(2)}$ and $\beta_0$.
Since it holds that $\|\beta_n^* - \beta_0\|_1 = o_p(1)$ by Theorem \ref{l_1}, 
we can see that
\[
\sup_{t \in [0,1]} \left\|H_n(\beta_n^*,t) + \int_0^t \frac{s_n^{(1)}}{s_n^{(0)}}(\beta_0,s) \lambda_0(s) ds\right\|_\infty
= o_p(1)
\]
by a similar way to the proof of Lemma \ref{fischer}.
Therefore, we obtain the following theorem, which is proved by using Slutsky's theorem and a similar way to that in \cite{key gill}.
\begin{thm}\label{Breslow}
Under Assumption \ref{regularity} and \ref{matrix}, it holds that
$\sqrt{n}(\hat{\beta}_{n\hat{T}_n}^{(2)}-\beta_{0T_0})1_{\{\hat{T}_n = T_0\}}$ and the process equal in the point $t$ to 
\[
\left[\sqrt{n}\{\hat{\Lambda}(t)-\Lambda_0(t)\} + \sqrt{n} 
\int_0^t (\hat{\beta}_{n\hat{T}_n}^{(2)}-\beta_{0T_0})^T \frac{s^{(1)}}{s^{(0)}}(\beta_{0T_0},s) \lambda_0(s) ds\right] 1_{\{\hat{T}_n = T_0\}} 
\]
are asymptotically independent. The latter process is asymptotically distributed as 
a Gaussian martingale with the variance function
\[
\int_0^t  \frac{\lambda_0(s)}{s^{(0)}(\beta_{0T_0},s)} ds.
\]
\end{thm}
\vskip 20pt
{\bf Acknowledgements.}
The author would like to express the appreciation to Prof.\ Y.\ Nishiyama of Waseda University and Dr.\ K.\ Tsukuda of the University of Tokyo for long hours discussion about this work.



\vskip 20pt

\end{document}